\documentclass[11pt,twoside]{amsart}
\usepackage{amsmath, amsthm, amscd, amsfonts, amssymb, graphicx, color}
\usepackage[bookmarksnumbered, plainpages]{hyperref}
\newtheorem{thm}{Theorem}[section]
\newtheorem{cor}[thm]{Corollary}
\newtheorem{lem}[thm]{Lemma}
\newtheorem{exa}[thm]{Example}
\newtheorem{prop}[thm]{Proposition}

\newtheorem{rem}[thm]{\bf{Remark}}

\numberwithin{equation}{section}
\def\pn{\par\noindent}

\begin{document}

\vspace{1.3 cm}


\title{The probability that the product of two elements of finite group algebra is zero}
\author{Haval M. Mohammed Salih$^*$}

\thanks{{\scriptsize
\hskip -0.4 true cm MSC(2010): Primary: 16S34; Secondary: 16U60.
\newline Keywords: group algebra, probability, Wedderburn decomposition.\\
}}
\maketitle



\begin{abstract}  Let $\mathbb{F}_qG$ be a finite group algebra. We denote by $P(\mathbb{F}_qG)$ the probability that the product of two elements of $\mathbb{F}_qG$ be zero. In this paper, the general formula for computing the $P(\mathbb{F}_qG)$ are established for the cyclic groups $C_n$, the Quaternion group $Q_8$ and the symmetric group $S_3$, for some cases.
\end{abstract}

\vskip 0.2 true cm


\section{\bf Introduction}
\vskip 0.4 true cm

Throughout this paper, $\mathbb{F}_q$ denotes a finite field with $q$ elements where $q$ is a prime power of $p$. Let $G$ be a finite non trivial group. Then $\mathbb{F}_qG$ is the group algebra of $G$ over the field $\mathbb{F}_q$. The set of all invertible elements of $\mathbb{F}_qG$ form a group called the unit group of $\mathbb{F}_qG$ denoted by $\mathcal{U}(\mathbb{F}_qG)$. There are many known results about the group unit of $\mathbb{F}_qG$ in \cite{makhijani2014unit, gildea2008order, gaohua2011unit}. The main topic of this study is to find the nullity degree (probability) of $\mathbb{F}_qG$, which is the probability that the multiplication of two randomly chosen elements of $\mathbb{F}_qG$ is zero. That is

 $$P(\mathbb{F}_qG)=\frac{|\{(a,b)\in \mathbb{F}_qG\times \mathbb{F}_qG \;: ab=0 \}|}{|\mathbb{F}_qG|^2}.$$
where $a:=\sum_{g\in G}r_g g$ and  $b:=\sum_{g\in G}\bar{r}_g g$  $(\bar{r}_g, r_g\in \mathbb{F}_q,g\in G)$. Since $\mathbb{F}_qG$ is a finite group ring with identity, then we have that $\mathbb{F}_qG=\{0\}\cup \mathcal{U}(\mathbb{F}_qG)\cup ZD(\mathbb{F}_qG)$, where $ZD(\mathbb{F}_qG)$ is the set of nonzero zero divisors of $\mathbb{F}_qG$. For $x\in \mathbb{F}_qG$, the left (right) annihilator of $x$ is the set $Ann_l(x)=\{\alpha \in \mathbb{F}_qG \;: \alpha x=0\}$ ($Ann_r(x)=\{\alpha \in \mathbb{F}_qG \;: x \alpha =0\}$). Similarly we define the annihilator of $x$ by $Ann(x)=\{\alpha \in \mathbb{F}_qG \;: \alpha x=x\alpha =0\}$. If $G$ is an abelian group, then we can write $P(\mathbb{F}_qG)$ in term of the unit group and zero divisor set as follows:

\begin{align}
	P(\mathbb{F}_qG) &= \frac{\sum_{x\in \mathbb{F}_qG}|Ann(x)|}{|\mathbb{F}_qG|^2} \\
	&= \frac{|\mathbb{F}_qG|+|\mathcal{U}(\mathbb{F}_qG)|+\sum_{0\neq x\in ZD(\mathbb{F}_qG)}|Ann(x)|} {|\mathbb{F}_qG|^2},
	\label{eq3}
\end{align}

If $G$ is non abelian group, then

$$P_l(\mathbb{F}_qG) = \frac{|\mathbb{F}_qG|+|\mathcal{U}(\mathbb{F}_qG)|+\sum_{0\neq x\in ZD(\mathbb{F}_qG)}|Ann_l(x)|} {|\mathbb{F}_qG|^2}$$
The probability of a finite commutative ring $R$ with identity can be found in \cite{esmkhani2018probability}. Recently in \cite{haval2021}, Mohammed Salih derive the general formula for computing $P(\mathbb{F}_qG)$ where $|G|<5$.

In this paper, $M_n(\mathbb{F}_q)$ denotes the ring of all $n\times n$ matrix over $\mathbb{F}_q$, $R$ denotes a ring with identity. Also $C_n=\langle a | a^n=1\rangle$ denotes the cyclic group of order $n$ and $R_1\bigoplus R_2$ denotes direct sum of rings $R_1$ and $R_2$. The group that we define below via its presentation is called the Quaternion group of order 8. $$Q_8=\langle a,b | a^4=1, a^2=b^2, bab^{-1}=a^{-1}\rangle.$$


\section{\bf Background}
\vskip 0.4 true cm
The following two results are well known \cite{makhijani2014unit}. These give us the structure of the unit group of some group algebras $\mathbb{F}_qC_n$.

\begin{thm}\label{tt1} \cite{makhijani2014unit}
	If $gcd(n,p)=1$ and $q=p^m$, then $\mathcal{U}(\mathbb{F}_qC_n)\cong C_{q-1}\times (\displaystyle\prod_{l|n, 
		l>1}C_{q^{d_l}-1}^{el})$
	where $d_l$ is the multiplicative order of $q$ modulo $l$ and $e_l=\frac{\varphi(l)}{d_l}$.
\end{thm}

Now consider the case $p|n$.
\begin{lem}\label{ll1} \cite{makhijani2014unit}
	Let $k\in \mathbb{N}$. Then 
	$\mathcal{U}(\mathbb{F}_{p^m}C_{p^k})\cong   \left\{  
	\begin{array}{ll}
	C_{p^m-1}\times	C_p^{m(p-1)} & if k=1 \\
	C_{p^m-1}\times \displaystyle \prod_{t=1}^kC_{p^t}^{n_t}  & \text{ otherwise} \\
	\end{array} 
	\right.$\\
	where $n_k=m(p-1)$ and $n_t=mp^{k-t-1}(p-1)^2$, for all $t, 1\leq t<k$.
\end{lem}

The next three results tell us the group algebra $KG$ has Wedderburn decomposition of the form.
\begin{prop} \cite{milies2002introduction}
Let $G$ be an abelian group of order $n$ and $K$ a field such that the characteristic of the field does not divide $n$. If $K$ contains a primitive root of unity of order $n$, then $KG\cong\underset{\text{n-times}}{K\bigoplus \dots \bigoplus K}$.
\end{prop}

\begin{prop} \cite{milies2002introduction}
	\label{p2}
Let $F$ be a finite field where the characteristic of the field is not equal to 2. Then $FQ_8=F\bigoplus F\bigoplus F\bigoplus F\bigoplus M_2(F)$ if and only if $x^2+y^2=-1$ can be solved in $F$.
\end{prop}

\begin{prop}\cite{martinez}
	Let $\mathbb{F}_q$ be a finite field and $gcd(q,6)=1$. Then $\mathbb{F}_qS_3\cong \mathbb{F}_q\bigoplus \mathbb{F}_q\bigoplus M_2(\mathbb{F}_q)$.
\end{prop}

\begin{thm}\cite{haval2021}
	Let $RG$ be a finite group ring. Then $P(RG)\geq \frac{1}{4}$ if and only if
	$RG$ is isomorphic to one of the following rings: $\mathbb{Z}_2C_2, \mathbb{Z}_3C_2, \mathbb{Z}_2C_3$. 
\end{thm}

Initially, it was challenging for us to derive a computing formula for $P(\mathbb{F}_qG)$, since there is no analogous concept of it. Finally, we are able to get the formula for $P(\mathbb{F}_qG)$, where $|G|<5$ in \cite{haval2021}. In this paper, it may be pointed out that for most of these group algebras, the GAP program fails to compute the $P(\mathbb{F}_qG)$ if $|G|\geq 6$, then the program may crash, even by using super computer. So we need to solve this problem by using algebraic tools. Further, we provide the explicit formula of $P(\mathbb{F}_qC_n)$, if $gcd(n,p)=1$. Also we find the $P(\mathbb{F}_qG)$ where $G=S_3$ or $Q_8$ for some cases by using the Wedderburn decomposition.

\section{The general formula of the probability of some group algebras $\mathbb{F}_qG$}
In this section, we get serveral results for $P(\mathbb{F}_qG)$ and derive a computing formula for $P(\mathbb{F}_qG)$. We begin with the following result:
\begin{lem}
	\label{l12}
Let $R_1,...,R_n$ be finite rings. Then $P(R_1\bigoplus R_2\bigoplus ...\bigoplus R_n)=P(R_1)P(R_2)\dots P(R_n)$.
\end{lem}
\begin{proof}
It is straightforward.
\end{proof}

\begin{thm}
	\label{th33}
If $gcd(n,p)=1$ and $q=p^m$, then $P(\mathbb{F}_qC_n)=(\frac{2q-1}{q^2}) \times \prod_{l|n\;l>1}(\frac{2q^{d_l}-1}{q^{2d_l}})^{e_l}$ where $d_l$ is the multiplicative order of $q$ module $l$ and $e_l=\frac{\phi(l)}{d_l}$, $\phi$ denotes the Euler Totient function.
\end{thm}

\begin{proof}
Using \cite[Theorem 1]{Perlis} and \cite[Theorem 2.21, pp.53]{Lidl}, it follows that $\mathbb{F}_qC_n \cong \mathbb{F}_q\bigoplus\;(\bigoplus_{l|n\;l>1} \mathbb{F}_{q^{d_l}}^{e_l})$. By Lemma \ref{l12} and \cite[Lemma 3.1]{esmkhani2018probability}, the result follows. 
\end{proof}

Note that, we are going to use Perlis-Walker Theorem in order to determine how many elements has annihilator of size $q^l$.
\begin{thm}\label{tt7}
	Let $\mathbb{F}_q$ be a field of characteristic $p$, where $q=p^m$. Then\\ 
	$P(\mathbb{F}_{q}C_{5})\cong   \left\{  
	\begin{array}{ll}
	\frac{q^7-q^6+q^5+q^4-q^2+q-1}{q^{9}} & if \;  p|5. \\
	\frac{4q^5-2q^4-2q+1}{q^{10}}  & q\equiv 2 \text{ or }3 \mod 5 \text{ and } if \; gcd(p,5)=1. \\
	\frac{q^6-2q^3+5q^2-2q-1}{q^{8}}  & q\equiv 4 \mod 5 \text{ and } if \; gcd(p,5)=1. \\
	\frac{2q^6+5q^5-21q^4+35q^3-29q^2+10q-1}{q^{10}}  & q\equiv 1 \mod 5 \text{ and } if \; gcd(p,5)=1. \\
	\end{array} 
	\right.$
\end{thm}

\begin{proof}
	Let $v\in ZD(\mathbb{F}_{q}C_{5})\setminus\{0\}$ and $\psi\colon \mathbb{F}_qC_5\rightarrow \mathbb{F}_qC_5 $ the $\mathbb{F}_q$-linear transformation defined $\psi(x)=vx$, then $Ann(x)=Ker \psi$ that a non trivial $\mathbb{F}_q$-linear subspace of dimension $1\leq l\leq 4$. That is $|Ann(x)|=q,q^2,q^3$ or $q^4$. We have four cases to analysis:
	
\begin{itemize} 
\item [Case 1,]
If $q\cong 0\mod 5$, then there are $q-1,q(q-1),q^2(q-1),q^3(q-1)$ elements $0\neq x\in ZD(\mathbb{F}_{q}C_{5})$ such that  $|Ann(x)|=q,q^2,q^3,q^4$. By Lemma \ref{ll1}, we obtain $|\mathcal{U}(\mathbb{F}_qC_5)|=q^4(q-1)$. 
	
\item [Case 2,] if $q \cong 2 \mod 5$ or $q \cong 3 \mod 5$, there are $q^4-1,0,0,(q-1)$ elements $0\neq x\in ZD(\mathbb{F}_{q}C_{5})$ such that  $|Ann(x)|=q,q^2,q^3,q^4$. Also, by Theorem \ref{tt1}, we obtain $|\mathcal{U}(\mathbb{F}_qC_5)|=(q^4-1)(q-1)$.

\item [Case 3,] if $q \cong 4 \mod 5$, then there are $q-1,2(q^2-1),2(q^2-1)(q-1),(q^2-1)^2$ elements $0\neq x\in ZD(\mathbb{F}_{q}C_{5})$ such that  $|Ann(x)|=q,q^2,q^3,q^4$. So $|\mathcal{U}(\mathbb{F}_qC_5)|=(q^2-1)^2q$ by Theorem \ref{tt1}.

\item [Case 4,] if $q \cong 1 \mod 5$, then there are $5(q-1)^4,(q-1)^4,(q-1)^3,5(q-1)$ elements $0\neq x\in ZD(\mathbb{F}_{q}C_{5})$ such that  $|Ann(x)|=q,q^2,q^3,q^4$. So $|\mathcal{U}(\mathbb{F}_qC_5)|=(q-1)^5$ by Theorem \ref{tt1}.
\end{itemize}
	The rest follows from Equation (\ref{eq3}).
\end{proof}

\begin{rem}
We can see that the proof of the cases 2,3 and 4 in Theorem \ref{tt7} follows from Theorem \ref{th33}.
\end{rem}
\begin{lem}
	\label{th12}
Let $RG$ be a finite group algebra with identity such that $P(RG)\geq \frac{1}{10}$. Then the structure of $RG$ and the possible values of $P(RG)$ are given in Table \ref{t1}.
\end{lem}
\begin{proof}
The proof follows from direct computations and \cite[Theorems 3.1, 3.2, 3.3 and 3.4]{haval2021}.
\end{proof}

A direct consequence of Lemma \ref{th12} is the following result.
\begin{cor}
Let $RG$ be a finite group algebras with identity. Then the following holds:
\begin{enumerate}
	\item $P(RG)\notin (\frac{21}{64},\frac{1}{4})$.
	\item If $RG\ncong \mathbb{F}_2C_2, \mathbb{F}_2C_3$, then $P(RG)\leq \frac{21}{64}$. Moreover, the equality holds if and only if $RG\cong \mathbb{F}_2C_3$.
	\item If $RG\ncong \mathbb{F}_3C_2, \mathbb{F}_2C_2, \mathbb{F}_2C_3$, then $P(RG)< \frac{2}{10}$.
	\item $P(RG)= \frac{7}{32}$ if and only if $RG\cong \mathbb{F}_2(C_2\times C_2), \mathbb{F}_5C_2$.
\end{enumerate}
\end{cor}

\begin{table}[h]
	\caption{All group algebras with $P(RG)\geq 0.1$}
\centering
	\begin{tabular}{ |c| c| c| }
		\hline
		$n$ & $RG$ & $P(RG)$ \\ 
		\hline
		1 & $\mathbb{F}_2C_2$ & $\frac{1}{2}=0.5$ \\  
		2 & $\mathbb{F}_3C_2$ & $\frac{25}{81}\simeq 0.308$   \\
		3 & $\mathbb{F}_5C_2$ & $\frac{81}{625}\simeq 0.129$   \\
		4 & $\mathbb{F}_2C_3$ & $\frac{21}{64}\simeq 0.328$   \\
		5 & $\mathbb{F}_2C_4$ & $\frac{3}{36}\simeq 0.18$   \\
		6 & $\mathbb{F}_3C_3$ & $\frac{1}{9}\simeq 0.111$   \\
		7 & $\mathbb{F}_4C_2$ & $\frac{5}{32}\simeq 0.156$   \\
		8 & $\mathbb{F}_2(C_2\times C_2)$ & $\frac{7}{32}\simeq 0.218$   \\
		9 & $\mathbb{Z}_4C_2$ & $\frac{7}{32}\simeq 0.218$   \\
		10 & $\mathbb{Z}_6C_2$ & $\frac{25}{162}\simeq 0.154$   \\
		11 & $\mathbb{F}_2S_3$ & $\frac{5}{64}\simeq 0.113$   \\
	
		\hline
	\end{tabular}
\label{t1}
\end{table}
\vskip 0.4 true cm

\begin{prop}
	\label{p3}
	If $R=M_{2}(\mathbb{F}_q)$ is the matrix ring algebra over $\mathbb{F}_q$ with two generators where $q=p^t$, then $P_l(R)=\frac{q^4+3q^3-2q(q+1)+1}{q^7}$.
	\label{th1}
\end{prop}
\begin{proof}
By the same idea as in proof of Theorem \ref{tt7}, we have $|Ann_l(x)|=q^2$ for all $0\neq x\in ZD(R)$. By using \cite[Theorem 3.1]{Cheraghpour}, we obtain $\mathcal{U}=q(q^3-q^2-q+1)$. Since $R$ is finite ring with identity, then $R=\{0\}\cup \mathcal{U}\cup ZD(R)$. So there are $q^3+q^2-q-1$ elements $0\neq x\in ZD(R)$ of size $|Ann(x)|=q^2$. So $\sum_{0\neq x\in ZD(R)}|Ann_l(x)|=q^2(q^3+q^2-q-1)$. Put these equations in Equation (\ref{eq3}) and the result follows.
\end{proof}

\begin{prop}
	\label{p4}
	If $R=M_{2}(\mathbb{F}_q)$ is the matrix ring algebra over $\mathbb{F}_q$ with two generators where $q=p^t$, then $P(R)=\frac{3q^2-2}{q^6}$.
	\label{th11}
\end{prop}
\begin{proof}
By the same idea as in proof of Theorem \ref{tt7}, we have $|Ann_l(x)|=q$ for all $0\neq x\in ZD(R)$. The rest is similar as the proof of Proposition \ref{p3}.
\end{proof}

\begin{thm}
	\label{th2}
Let $\mathbb{F}_q$ be a finite field of characteristic $p\neq 2$. Then 
	\begin{enumerate}
		\item  $P_l(\mathbb{F}_qQ_8)=(\frac{2q-1}{q^2})^4\times (\frac{q^4+3q^3-2q^2-2q+1}{q^7})$.
		\item $P(\mathbb{F}_qQ_8)=(\frac{2q-1}{q^2})^4\times (\frac{3q^2-2}{q^6})$.
	\end{enumerate}
\end{thm}
\begin{proof}
	
Since $\mathbb{F}_q$ is a finite field of characteristic $p\neq 2$. Then by Proposition \ref{p2}, we have $\mathbb{F}_qQ_8=\mathbb{F}_q\bigoplus \mathbb{F}_q\bigoplus \mathbb{F}_q\bigoplus \mathbb{F}_q\bigoplus M_2(\mathbb{F}_q)$. 
\begin{enumerate}
\item From Lemma \ref{l12}, \cite[Lemma 3.1]{esmkhani2018probability} and Proposition \ref{p3}, the result follows.
\item From Lemma \ref{l12}, \cite[Lemma 3.1]{esmkhani2018probability} and Proposition \ref{p4}, the result follows.
\end{enumerate}

\end{proof}

\begin{thm}
	\label{th22}
	Let $\mathbb{F}_q$ be a finite field and $gcd(q,6)=1$. Then
	\begin{enumerate}
		\item  $P_l(\mathbb{F}_qS_3)=(\frac{2q-1}{q^2})^2\times (\frac{q^4+3q^3-2q^2-2q+1}{q^7})$.
		\item $P(\mathbb{F}_qS_3)=(\frac{2q-1}{q^2})^2\times (\frac{3q^2-2}{q^6})$.
	\end{enumerate}
\end{thm}
\begin{proof}
The proof is similar as Theorem \ref{th2}.
\end{proof}

\begin{prop}
Let $\mathbb{F}_q$ be a finite field of characteristic 2. Then 
\begin{enumerate}
\item $P_l(\mathbb{F}_qS_3)=\frac{3q^5+7q^4-12q^3-2q^2+7q-2}{q^{10}}$.
\item  $P(\mathbb{F}_qS_3)=\frac{9q^3-6q^2-6q+4}{q^{9}}$.
\item  $P(\mathbb{F}_qQ_8)=\frac{3q^2+3q-5}{q^9}$.
\end{enumerate}
\end{prop}
\begin{proof}
Clear.
\end{proof}

\begin{exa}
Consider the finite field $\mathbb{F}_7$, the symmetric group $S_3$ and the cyclic group $C_6$. The probability $P_l(\mathbb{F}_7S_3)$ and $P(\mathbb{F}_7C_6)$ cannot compute by personal computer. So we use super computer with property \textbf{Intel®Xeon®Platinum 8280 CPU@ 2.70GHz 2.69GHz Installed RAM  128GB (128GB usable)}. The computation takes nearly 129 hours (see below) and the outputs are the following: 
\begin{verbatim}
[ rec(Size := [ 72576, 24192, 15840, 4608, 420, 12, 1 ],
    |ann_l|:=[ 1, 7, 49, 343, 2401, 16807, 117649 ], group := "S3", p := 560911/1977326743),
rec(Size:= [ 46656, 46656, 19440, 4320, 540, 36, 1 ],
   |ann|:=[ 1, 7, 49, 343, 2401, 16807, 117649 ], group := "C6", p := 4826809/13841287201)]
gap> time;
463879313
StringTime(463879313)
"128:51:19.313"
\end{verbatim}

On the other hand,  we achieve this computation with a few seconds by using Theorem \ref{th22}, part 1 and Theorem \ref{th33}.
\end{exa}

\begin{center}{\textbf{Acknowledgments}}
\end{center}
I would like thank to the referee for careful reading of the article and detailed report including corrections and comments; and I appreciate his/her effort on reviewing the article. \\ \\
\vskip 0.4 true cm



\bigskip
\bigskip

{\footnotesize \pn{\bf Haval M. Mohammed Salih}\; \\ {Department of
Mathematics, Faculty of Science}, {Soran University
, Kawa St, Soran,} {Erbil, Iraq}\\
{\tt Email: havalmahmood07\@gmail.com}\\

\end{document}